\documentclass[11pt,english]{article}
\usepackage[charter]{mathdesign}
\usepackage[T1]{fontenc}
\usepackage[utf8]{inputenc}
\usepackage[letterpaper]{geometry}
\usepackage{fancyhdr}
\usepackage{babel}
\usepackage{amsmath}
\usepackage{amsthm}
\usepackage[all]{xy}
\usepackage[unicode=true,
 bookmarks=true,bookmarksnumbered=false,bookmarksopen=false,
 breaklinks=false,pdfborder={0 0 1},backref=false,colorlinks=false]
 {hyperref}

\makeatletter
\numberwithin{equation}{section}
\numberwithin{figure}{section}
\theoremstyle{plain}
\newtheorem{thm}{\protect\theoremname}
  \theoremstyle{plain}
  
  \theoremstyle{definition}
  \newtheorem{defn}[thm]{\protect\definitionname}
  \theoremstyle{plain}
  \newtheorem{cor}[thm]{\protect\corollaryname}
  \theoremstyle{plain}
  \newtheorem{prop}[thm]{\protect\propositionname}

\newcommand\xyR[1]{\xydef@\xymatrixrowsep@{#1}}
\newcommand\xyC[1]{\xydef@\xymatrixcolsep@{#1}}

\newcommand\thistheoremname{}
\newtheorem{genericthm}[thm]{\thistheoremname}
\newenvironment{namedthm}[1]
{\renewcommand\thistheoremname{#1}\begin{genericthm}}
{\end{genericthm}}

\addto\captionsenglish{\renewcommand\bibname{References}}

\date{}

\makeatother

  \providecommand{\conjecturename}{Conjecture}
  \providecommand{\corollaryname}{Corollary}
  \providecommand{\definitionname}{Definition}
  \providecommand{\propositionname}{Proposition}
\providecommand{\theoremname}{Theorem}

\begin{document}
\lhead{On Contraction of Algebraic Points}\rhead{Bogomolov and Qian}

\title{\textsc{On Contraction of Algebraic Points}}

\author{\textsc{Fedor Bogomolov} and \textsc{Jin Qian}}
\maketitle
\begin{quote}
\textsc{\small{}Abstract.}{\small{} We study contraction of points on
$\mathbb{P}^1(\bar{\mathbb{Q}})$ with certain control
on local ramification indices, with application to the unramified curve correspondence
problem initiated by Bogomolov and Tschinkel.}{\small \par}
\end{quote}

\section{Introduction}

In this paper we address the following problem: let $P$ be a subset of natural numbers and
$S_1$, $S_2$ be two subsets of points on $\mathbb{P}^1(\bar{\mathbb{Q}})$.
We say $S_1$ can be $P$-contracted to $S_2$ if there is a rational map $f: \mathbb{P}^1 \to \mathbb{P}^1$ such
that the image of $S_1$ under $f$ and all branch points of $f$ are contained in $S_2$ with all local ramification indices of $f$ belonging to $P$.

\

One motivation of our problem is coming from Belyi's theorem. In this language Belyi's theorem states
that if $P$ is the set of all natural numbers, then any finite subset $S_1\subset \mathbb{P}^1(\bar{\mathbb{Q}})$ 
can be $P$-contracted to $S_2=(0,1,\infty)$ or to any three points in $\mathbb{P}^1(\bar{\mathbb{Q}})$.

\

Another motivation is coming from the study of unramified correspondences between curves.
Following \cite{3}, we make the following definition:

\medskip{}

\begin{defn}By a curve, we mean a smooth projective curve over $\bar{\mathbb{Q}}$. 
When we write an affine equation for a curve, its smooth projective model is understood.
If $C\rightarrow C''$ and $C'\rightarrow C''$ are surjective morphisms of curves, 
by a compositum of $C$ and $C'$ over $C''$, we mean a curve whose function field is
a compositum of $\bar{\mathbb{Q}}(C)$ and $\bar{\mathbb{Q}}(C')$ over $\bar{\mathbb{Q}}(C'')$.
By an unramified cover of $C$, we mean a curve $\tilde{C}$ together with an $\acute{e}$tale 
morphsim $\tilde{C}\rightarrow C$. Let $C$, $C'$ be two curves. We call
$C$ lies over $C'$ and write $C \Rightarrow C'$ if there exists an unramified cover of $C$ which
admits a surjective map to $C'$. If $C$ lies over $C'$ and $C'$ also lies over $C$, we call $C$ and $C'$ are equivalent and write $C \Leftrightarrow C'$.  Finally, denote by $\mathsf{C}_n$ the curve:
$y^2=x^n-1$.

\end{defn}

\

In the study of such correspondence, an important step which is closely related to our contraction problem is
the construction of unramfied covers for which we need to find maps from various intermediate
curves to $\mathbb{P}^1$ or some elliptic curves with restrictions on local ramification indices
and the number of branch points. This method was established by Bogomolov and Tschinkel in \cite{3} where they have showed that any hyperbolic hyperelliptic curve lies over $\mathsf{C}_6$.

\

Here in section 2, our main results are:

\begin{thm}
If the only prime divisors of $n$ and $m$ are 2, 3 and 5, then
$\mathsf{C}_n \Leftrightarrow \mathsf{C}_m$ and for any $k\geq 5$ we have 
$\mathsf{C}_k \Rightarrow \mathsf{C}_n$.

\end{thm}

\begin{namedthm}{Remark}
\textup{Although Theorem 2 is also established in \cite{3}, the proof contains several gaps in the
construction of unramified covers. Based on the idea in \cite{3}, here we will use a different approach to establish this result.}
\end{namedthm}

\begin{thm}
If $n=2^a3^b5^c7^d$, then $\mathsf{C}_{6\cdot 13^d} \Rightarrow \mathsf{C}_n$.
\end{thm}

\

In \cite{3}, Bogomolov and Tschinkel have made the conjecture that the curve
$\mathsf{C}_6$ lies over any other curve. The reason why we are interested in the family of curves
\{$\mathsf{C}_n$\} is that the Bogomolov-Tschinkel conjecture will hold if $\mathsf{C}_6$ lies over
$\mathsf{C}_n$ for any positive integer $n$ (See Proposition 17). Towards this conjecture, in section 2 and section 3 we introduce the notion of contracting a finite given subset of $\mathbb{P}^1(\bar{\mathbb{Q}})$
into another finite subset of $\mathbb{P}^1(\bar{\mathbb{Q}})$ with restrictions on the local ramification indices (See Definition 15) and the notion of contracting a finite subset of $\mathbb{P}^1(\bar{\mathbb{Q}})$ to a four-point subset of $\mathbb{P}^1(\bar{\mathbb{Q}})$ via elliptic curves
 (See Definition 31). We have obtained some criterions for a curve $C$ with $\mathsf{C}_6$ lying over $C$ (See Theorem 16, Theorem 32, Corollary 33). In section 4, we will propose a procedure to approach the
 Bogomolov-Tschinkel conjecture.

\

\section{\label{section 2} Unramified Correspondences over $\bar{\mathbb{Q}}$}

{\bf{Notations}}. Let $f: C\rightarrow C'$ be a surjective morphism of curves. We denote by Bran($f$) the branch locus of $f$ and denote by Ram($f$) the ramification points of $f$. For a point $y
\in$ Bran($f$), $x\in f^{-1}(y)$, denote by $e(x|y)$ the local ramification index of $x$ at $y$. For a set of four points $a,b,c,d \in  \bar{\mathbb{Q}}$, we denote by $E(a,b,c,d)$ an elliptic curve branched over $\{a,b,c,d\}$. 

\

In this section, we will establish some results about the unramified curve correspondence problem.
The key tool is:

\

\noindent {\bf{Abyhankar's Lemma}}. Let $f: C\rightarrow C''$ and $g: C'\rightarrow C''$ be surjective morphisms of curves. Denote by $\hat{C}$ the compositum of $C$ and $C'$ over $C''$ with corresponding map $h$
and $l$:
\[
\xymatrix{ & \hat{C} \ar[dl]^{h} \ar[dr]^{l} \\
C \ar[dr]^{f} & & C'\ar[dl]^{g} \\
 & C''
}
\]

\noindent Assume $x\in C$ and $y\in C'$ such that $f(x)=g(y)=z$ for some point $z$ on $C''$. 
Suppose $f^{-1}(z) = \{x_1, ... , x_s\}$, $g^{-1}(z) = \{y_1, ... , y_t\}$ and denote by $d$ the greatest common
divisor of $e(x_i|z)$ for $i=1, ... ,s$. If for any $j$, we have:
\begin{equation*}
e(y_j|z) \ |\ d.
\end{equation*}
Then for any $i$, $x_i$ is unramified under $h$ and for any $j$ and
any point $a\in l^{-1}(y_j)$ we have:
\begin{equation*}
e(a|y_j)=\frac{e(h(a)|z)}{e(y_j|z)}.
\end{equation*}

\noindent In particular, if for all points $x\in C$ and $y\in C'$ with $f(x)=g(y)$ we have:
\begin{equation*}
e(y|g(y)) \ | \  e(x|f(x)).
\end{equation*}
Then a compositum of $C$
and $C'$ over $C''$ is an unramified cover of $C$. 

\begin{proof}
This follows from Theorem 3.9.1 in \cite{10}.
\end{proof}

\

In the following proofs, our main strategy is to construct the unramified covers of curves directly via Abhyankar's lemma. In order to make such constructions using Abhyankar's lemma, we will explicitly contract some cyclotomic roots and also use some special elliptic curves to contract and spread points. 

\

\begin{prop}
 Let $H$ be a hyperbolic hyperelliptic curve. Then $H \Rightarrow \mathsf{C}_6$.
\end{prop}

\begin{proof}
This is one part of Propostition 2.4 in \cite{3}.
\end{proof}

\begin{prop} 
Let $H$ be a hyperbolic hyperelliptic curve. Then $H \Rightarrow \mathsf{C}_8$.
\end{prop}

\begin{proof}
Consider the following diagrams:
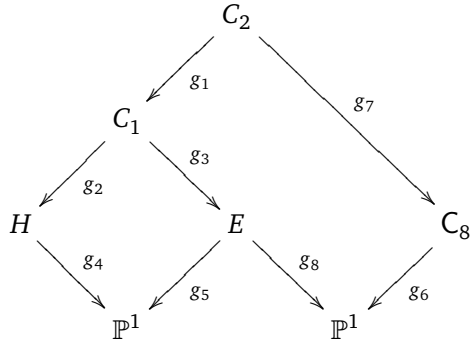
\[
\xymatrix{ & & C_2 \ar[ld]^{g_1} \ar[rrdd]^{g_7}\\
 & C_1 \ar[ld]^{g_2} \ar[rd]^{g_3} \\
 H \ar[rd]^{g_4} & & E \ar[ld]^{g_5} \ar[rd]^{g_8}& & \mathsf{C}_8 \ar[ld]^{g_6} \\
 & \mathbb{P}^1 & &\mathbb{P}^1 \\
}
\]

and

\[
\xymatrix{ \mathsf{C}_8 \ar[r]^{f_1}&  \mathbb{P}^1 \ar[r]^{f_2}& \mathbb{P}^1 \ar[r]^{f_3}& \mathbb{P}^1}.
\]
\

Denote $f_3 \circ f_2 \circ f_1$ by $f$.
In these diagrams:

(i) The map $f_1$ is the standard degree 2 projection with Bran$(f_1)$ containing all 8th roots
of unity with local ramification indices being 2;

(ii) The map $f_2$ is $x^4$;

(iii) The map $f_3$ is $(\frac{x-1}{x+1})^2$;

(iv) The map $g_4$ is the standard degree 2 projection which has 6 branch points;

(v) The map $g_6$ is $f$. Bran($g_6$)=$\{0,1,\infty\}$ with all local ramification indices being 4;

(vi) $E$ is an elliptic curve branched at 4 points of Bran$(g_4)$;

(vii) The map $g_8$ is the standard degree 2 projection combined with an automorphism
of $\mathbb{P}^1$
such that Bran$(g_8)$ contains $\{0,1,\infty\}$;

(viii) The map $g_5$ is a composition of a multiplication-by-2 map, a translation-by-$R$ map
and the standard degree 2 projection such that the image of $R$ under the
standard degree 2 projection is a point in Bran$(g_4)$
which is different from the 4 points in (vi);

(ix) The curve $C_1$ is a compositum of $H$ and $E$ over $\mathbb{P}^1$. 
Since Bran$(g_5)$ consists of the image of two-torsion points of $E$ under the standard projection, by (iv) and (v) we see that
$C_1$ is an unramified cover of $H$;

(x) The curve $C_2$ is a compositum of $C_1$ and $\mathsf{C}_8$ over $\mathbb{P}^1$. Note that in (viii) all two torsion points of $E$ are mapped to a point in Bran$(g_4)$
which is different from the 4 points in (vi). By Abyhankar's lemma, these points are
in the branch locus of $g_3$ with local ramification indices being 2. Thus,
Bran$(g_8 \circ g_3)$ contains 0,1 and $\infty$ with local ramification indices being 4.
By Abhyankar's lemma, we have:
$C_2$ is an unramified cover of $C_1$. Combined with (ix), we see that $C_2$ is an unramified cover of $H$ which maps surjectively onto $\mathsf{C}_8$.

\end{proof}

\begin{prop}
$\mathsf{C}_{8n} \Rightarrow \mathsf{C}_{16n}$ and 
$\mathsf{C}_{16n} \Rightarrow \mathsf{C}_{24n}$  \ for $n \geq 1$.
\end{prop}

\begin{proof}
First, let us show: $\mathsf{C}_{8n} \Rightarrow \mathsf{C_{16n}}$ for $n \geq 1$:

Consider the following diagrams:
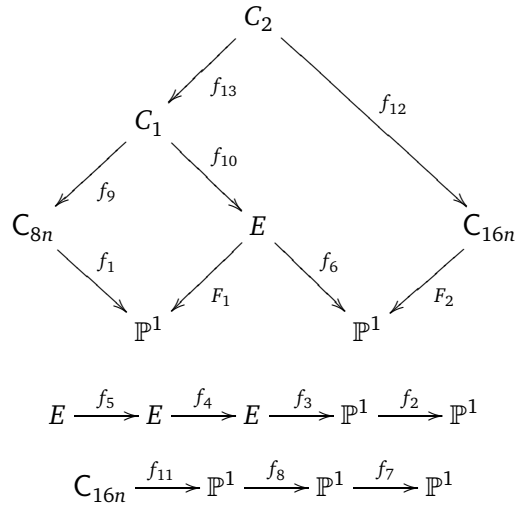
\[
\xymatrix{& & C_2\ar[dl]^{f_{13}} \ar[ddrr]^{f_{12}} & & \\
 & C_1\ar[dl]^{f_9} \ar[dr]^{f_{10}} \\
 \mathsf{C}_{8n}\ar[dr]^{f_1} & & E\ar[ld]^{F_1} \ar[dr]^{f_6} & &\mathsf{C}_{16n} \ar[dl]^{F_2}\\
  & \mathbb{P}^1 & & \mathbb{P}^1\\}
\]
\[
\xymatrix{E\ar[r]^{f_5}& E\ar[r]^{f_4}& E\ar[r]^{f_3} &\mathbb{P}^1 \ar[r]^{f_2} 
&\mathbb{P}^1 \\ }
\]
\[
\xymatrix{\mathsf{C}_{16n}\ar[r]^{f_{11}}& \mathbb{P}^1 \ar[r]^{f_8} 
&\mathbb{P}^1 \ar[r]^{f_7} & \mathbb{P}^1 \\ }
\]

\

In these diagrams:

(i) The map $f_2$ is $x^2$;

(ii) The curve $E$ is defined by: $y^2=x^3-x$ and $f_3$ is the standard projection;

(iii) The map $f_4$ is the translation-by-$R$ map where $R=(1,0)$;

(iv) The map $f_5$ is the multiplication-by-2 map;

(v) The map $F_1$ is $f_2 \circ f_3 \circ f_4 \circ f_5$;

(vi) The map $f_6$ is the standard degree 2 projection;

(vii) The map $f_7$ is $(\frac{x-1}{x+1})^2$; 

(viii) The map $f_8$ is $x^{8n}$;

(ix) The map $f_{11}$ is the standard degree 2 projection;

(x) The map $F_2$ is $f_7\circ f_8 \circ f_{11}$. Bran($F_2$) $=\{0,1,\infty\}$
with corresponding local ramification indices being 4, 8$n$, 4;

(xi) The map $f_1$ is the projection $y$ composed with an automorphism of $\mathbb{P}^1$ which maps three 
branch points to 0,1 and $\infty$ such that $f_1^{-1}(1)$ and $f_1^{-1}(\infty)$ each contains one point
with ramification index 8n and $f_1^{-1}(0)$ contains two points with ramification indices 4n;

(xii) The curve $C_1$ is a compositum of $\mathsf{C}_{8n}$ and $E$ 
over $\mathbb{P}^1$ (via map $f_1$ and $F_1$);

(xiii) The curve $C_2$ is a compositum of $C_1$ and $\mathsf{C}_{16n}$ over
$\mathbb{P}^1$ (via map 
$f_6 \circ f_{10}$ and $F_2$).

\

We see that:

(1) Since Bran($F_1$)=($0,1,\infty$) with local ramification indices: $2, 4, 4$ (over $1, 0, \infty$ 
respectively), combined with (xi) we get: $f_9$ is unramified and each point of 
$F_1^{-1}(1)$ has ramification index $4n$ under $f_{10}$. Note that: 
$E[2]$ is contained in $F_1^{-1}(1)$;

(2) By (1), Bran($f_6\circ f_{10}$)=($0,1,-1,\infty$) with all local ramification 
indices being $8n$;
 
(3) By (2) and (x), $f_{13}$ is unramified. Combined with (1) we have:
\begin{equation*}
 \mathsf{C}_{8n} \Rightarrow \mathsf{C_{16n}}.
 \end{equation*}

\medskip{}
Next let us show: $\mathsf{C}_{8n} \Rightarrow \mathsf{C}_{12n}$ for $n \geq 1$ and $n$ even:

(which is the same as $\mathsf{C}_{16n} \Rightarrow \mathsf{C}_{24n}$ for $n \geq 1$)

Consider the following diagrams:

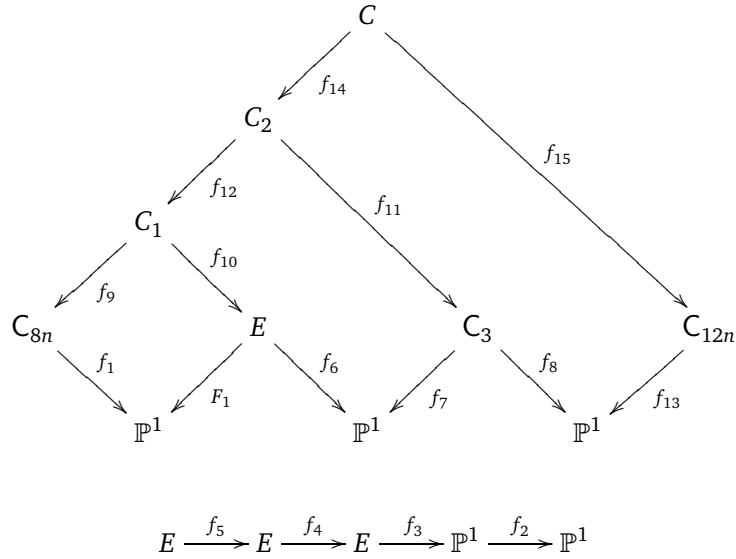
\[
\xymatrix{ & & & C\ar[dl]^{f_{14}} \ar[dddrrr]^{f_{15}} \\
& & C_2\ar[dl]^{f_{12}} \ar[ddrr]^{f_{11}} & & \\
 & C_1\ar[dl]^{f_9} \ar[dr]^{f_{10}} \\
 \mathsf{C}_{8n}\ar[dr]^{f_1} & & E\ar[ld]^{F_1} \ar[dr]^{f_6} & &\mathsf{C}_{3} \ar[dl]^{f_7}
 \ar[dr]^{f_8}
 & & \mathsf{C}_{12n}\ar[dl]^{f_{13}} \\
  & \mathbb{P}^1 & & \mathbb{P}^1 & & \mathbb{P}^1\\}
\]

\

\[
\xymatrix{E\ar[r]^{f_5}& E\ar[r]^{f_4}& E\ar[r]^{f_3} &\mathbb{P}^1 \ar[r]^{f_2} 
&\mathbb{P}^1 \\ }
\]

\

In these diagrams:

(i) The map $f_2$ is $x^2$;

(ii) The curve $E$ is defined by: $y^2=x^3-x$ and $f_3$ is the standard projection;

(iii) The map $f_4$ is the translation-by-$R$ map where $R=(1,0)$;

(iv) The map $f_5$ is the multiplication-by-3 map;

(v) The map $f_6$ is the standard projection combined with an automorphism of $\mathbb{P}^1$ such that:
$f_6(E[3])$ is the union of one point (this point is denoted by $a$) from Bran($f_6$) and ($1,\zeta _3, \zeta _3^2, \infty$);

(vi) The map $f_7$ is the multiplication-by-3 map combined with the standard projection; 

(vii) The map $f_8$ is $y$ combined with an automorphim of $\mathbb{P}^1$ which maps the three branch points 
to $0,1,\infty$;

(viii) The map $f_1$ is the projection $y$ composed with an automorphism of $\mathbb{P}^1$ which maps three 
branch points to 0,1 and $\infty$ such that $f_1^{-1}(1)$ and $f_1^{-1}(\infty)$ each contains one point
with ramification index 8n and $f_1^{-1}(0)$ contains two points with ramification indice 4n;

(ix) The curve $C_1$ is a compositum of $\mathsf{C}_{8n}$ and $E$
over $\mathbb{P}^1$ (via $f_1$ and $f_2
\circ f_3 \circ f_4\circ f_5$);

(x) The curve $C_2$ is a compostitum of $C_1$ and $\mathsf{C}_{3}$ over 
$\mathbb{P}^1$ (via 
$f_6 \circ f_{10}$ and $f_7$);

(xi) The curve $C$ is a compositum of $C_2$ and $\mathsf{C}_{12n}$ (via $f_8\circ f_{11}$ and $f_{13}$);

(xii) The map $f_{13}$ is the standard projection to $\mathbb{P}^1$ composed with $x^{6n}$ and $(\frac{x-1}{x+1})^2$.
Bran($f_{13}$)=($0,1,\infty$) and the corresponding ramification indices 
are ($4, 6n, 4$).

\

We see that:

(1) Since Bran($f_2 \circ f_3 \circ f_4 \circ f_5$)=($0,1,\infty$) with ramification indices: $2, 4, 4$(over $1, 0, \infty$ 
respectively), combined with (viii) we get: $f_9$ is unramified and each point of 
$(f_2 \circ f_3 \circ f_4 \circ f_5)^{-1}(1)$ has ramification index $4n$ under $f_{10}$. Note that: 
$E[3]$ is contained in $(f_2 \circ f_3 \circ f_4 \circ f_5)^{-1}(1)$;

(2) By (v) and (1), Bran($f_6\circ f_{10}$)=($a,1,\zeta _3, \zeta _3^2, \infty$) with local ramification 
indices being $4n$(over $1,\zeta _3, \zeta _3^2, \infty$) and $8n$(over $a$);

(3) By (vi) and (2), $f_{12}$ is unramified and $\mathsf{C}_3[3] \subseteq f_7^{-1}(1,\zeta _3, \zeta _3^2, \infty)$ which has ramification indices $2n$ under $f_{11}$;

(4) By (vii) and (3), ($0,1,\infty$) $\subset$ Bran($f_8 \circ f_{11}$) and they have local ramification 
indices $6n$;

(5) By (4) and (xii), we know that $C$ is an unramified cover of $C_2$ and hence we have:
\begin{equation*}
\mathsf{C}_{8n} \Rightarrow \mathsf{C}_{12n}
\end{equation*}
 for $n \geq 1$ and $n$ is even 
which is the same as 
\begin{equation*}
\mathsf{C}_{16n} \Rightarrow \mathsf{C}_{24n}
\end{equation*}
 for $n \geq 1$. 
\end{proof}

\begin{cor}
If $n \geq 6$ and the only prime divisors of $n$ are
2 and 3, then $\mathsf{C}_6 \Rightarrow \mathsf{C}_n$.
\end{cor}

\begin{proof}
Write $n=2^s3^t$, we have: (repeat applying Proposition 7) 
\begin{equation*}
\mathsf{C}_{6} \Rightarrow \mathsf{C_{8}} \Rightarrow 
\mathsf{C}_{16} \Rightarrow \mathsf{C}_{16\cdot 2} \Rightarrow
\mathsf{C}_{16\cdot 2^2} \Rightarrow ...... \Rightarrow \mathsf{C}_{16\cdot 2^{s+t}} \Rightarrow
\mathsf{C}_{16\cdot 2^{s+t-1}\cdot 3} \Rightarrow
\end{equation*}

\begin{equation*}
\mathsf{C}_{16\cdot 2^{s+t-2}\cdot 3^2} \Rightarrow ...... \Rightarrow
\mathsf{C}_{16\cdot 2^s\cdot 3^t} \Rightarrow \mathsf{C}_{2^s\cdot 3^t}=\mathsf{C}_n.
\end{equation*}

\end{proof}

\begin{prop}
$\mathsf{C}_6 \Rightarrow \mathsf{C}_5$.
\end{prop}

\begin{proof}
By Abhyankar's Lemma and Corollary 8, we only need to exhibit a map from $\mathsf{C}_5$ to $\mathbb{P}^1$ such that the branch points are exactly (0,1,$\infty$) and all local ramification indices have only prime divisors 2 or 3.

\

\noindent Consider the following maps:
\[
\xymatrix{ \mathsf{C}_5 \ar[r]^{f_1}&  \mathbb{P}^1 \ar[r]^{f_2}& \mathbb{P}^1 \ar[r]^{f_3}& \mathbb{P}^1\ar[r]^{f_4}
& \mathbb{P}^1 \ar[r]^{f_5}& \mathbb{P}^1 \ar[r]^{f_6}&  \mathbb{P}^1 \ar[r]^{f_7}&\mathbb{P}^1 \ar[r]^{f_8}
&\mathbb{P}^1 \ar[r]^{f_9}&\mathbb{P}^1 \\}
\]

\

Here: ($\zeta_5$ is denoted by $t$)

(i) The map $f_1$ is the degree 2 projection.

Bran($f_1$)=($1,t,t^2,t^3,t^4,\infty$) and all ramification indices are 2;

(ii) The map $f_2$ is $z+\frac{1}{z}$.

Ram($f_2$)=($1,-1$) with all ramification indices 2 and

Bran($f_2$)$\cup f_1$(Bran($f_1$))=($2,-2,t+t^4,t^2+t^3,\infty$). This set is denoted by $B_2$;

(iii) The map $f_3$ is $-\frac{1}{z}$.

$f_3$ is clearly unramified and $f_3(B_2)=(-\frac{1}{2},\frac{1}{2}, t^2+t^3,t+t^4,0)$. 
This set is denoted by $B_3$;

(Note that $(t+t^4)(t^2+t^3)=t^3+t^4+t+t^2=-1$.)

(iv) The map $f_4$ is $z^2+z-1$.

Ram($f_4$)=(-$\frac{1}{2},\infty$) with all ramification indices 2 and

Bran($f_4$) $\cup f_4(B_3)=(-\frac{5}{4}, \infty, -\frac{1}{4},0,-1)$. 
This set is denoted by $B_4$;

(v) The map $f_5$ is $-4z$.

Clearly it is unramified and $f_5(B_4)=(0,1,4,5,\infty)$.
This set is denoted by $B_5$;

(vi) The map $f_6$ is $4(z-\frac{5}{2})^2$.

Ram($f_6$)=($\frac{5}{2}, \infty$) with all ramification indices 2 and

Bran($f_6$) $\cup f_6(B_5)=(0,\infty,25,9)$. 
This set is denoted by $B_6$;

(vii) The map $f_7$ is $\frac{1}{2}(\frac{1}{2}(z+\frac{225}{z})+15)$.

Ram($f_7$)=($15,-15$) with all ramification indices 2 and

Bran($f_7$) $\cup f_7(B_6)=(0,15,16,\infty)$. 
This set is denoted by $B_7$;

(viii) The map $f_8$ is $\frac{z}{z-15}$.

Clearly it is unramified and $f_8(B_7)=(0,1,16,\infty)$. 
This set is denoted by $B_8$;

(ix) The map $f_9$ is $\frac{(z-1)^{32}\cdot (z-16)^3}{(z-10)^8\cdot z^{27}}$.

Ram($f_9$)=($0,1,10,16,\infty$) with corresponding ramification indices $3^3, 2^5, 2^3, 3, 3$ and

(Note that $\frac{df_9}{f_9}=\frac{4320}{z(z-1)(z-10)(z-16)}$ and the computation 
for ramification index of $\infty$ follows from the Riemann-Hurwitz Formula.)

Bran($f_9$) $\cup f_9(B_8)=(0,1,\infty)$.

\

By the computations in (i)-(ix), we see that Bran($f_9 \circ f_8 \circ f_7 \circ f_6 \circ f_5 \circ f_4 \circ f_3 \circ f_2 \circ f_1)=$ Bran($f_9$) $\cup f_9(B_8)=(0,1,\infty)$ with all local ramification indices only having prime divisors 2 or 3 (Note that in each step, the local ramification indices only have prime divisors 2 or 3).  

\end{proof}

\begin{prop}
$\mathsf{C}_{2^{11}\cdot 3^3 \cdot n} \Rightarrow \mathsf{C}_{5n}$
for $n \geq 1$.
\end{prop}

\begin{proof}
Let us still use this diagram:
\[
\xymatrix{ \mathsf{C}_5 \ar[r]^{f_1}&  \mathbb{P}^1 \ar[r]^{f_2}& \mathbb{P}^1 \ar[r]^{f_3}& \mathbb{P}^1\ar[r]^{f_4}
& \mathbb{P}^1 \ar[r]^{f_5}& \mathbb{P}^1 \ar[r]^{f_6}&  \mathbb{P}^1 \ar[r]^{f_7}&\mathbb{P}^1 \ar[r]^{f_8}
&\mathbb{P}^1 \ar[r]^{f_9}&\mathbb{P}^1 \\}
\]

\noindent Here $f_i$ are the maps as in the last proposition
and let $f=f_9 \circ f_8 \circ f_7 \circ f_6 \circ f_5 \circ f_4 \circ f_3 \circ f_2 \circ f_1$. Note that $f$ is a Belyi map from $\mathsf{C}_5$ to $\mathbb{P}^1$ with all local ramification indices divides $2^{10}\cdot 3^3$.

\medskip{}

Now let us consider the following diagram:
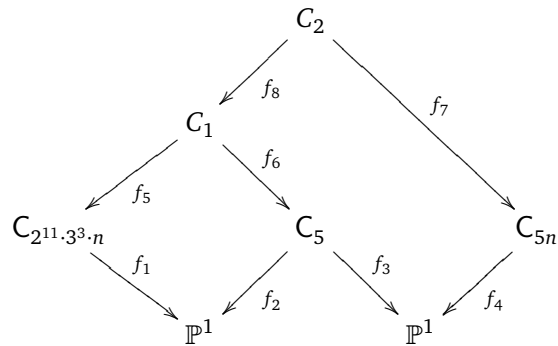
\[
\xymatrix{& & C_2\ar[dl]^{f_8} \ar[ddrr]^{f_7} & & \\
 & C_1\ar[dl]^{f_5} \ar[dr]^{f_6} \\
 \mathsf{C}_{2^{11}\cdot 3^3\cdot n}  \ar[dr]^{f_1} & & \mathsf{C}_5\ar[ld]^{f_2} \ar[dr]^{f_3} & &\mathsf{C}_{5n} \ar[dl]^{f_4}\\
  & \mathbb{P}^1 & & \mathbb{P}^1\\}
\]

\

\noindent In this diagram:

(i) The map $f_1$ is the projection $y$ composed with an automorphism of $\mathbb{P}^1$ which maps three 
branch points to 0,1 and $\infty$ such that $f_1^{-1}(0)$ and $f_1^{-1}(1)$ each contains one point
with ramification index $2^{11}\cdot 3^3\cdot n$ and $f_1^{-1}(\infty)$ contains two points with ramification
indice $2^{10}\cdot 3^3\cdot n$.

(ii) The map $f_2$ is the map $f$ above.

(iii) The curve $C_1$ is a compositum of  $\mathsf{C}_{2^{11}\cdot 3^3 \cdot n}$ and  $\mathsf{C}_{5}$
over $\mathbb{P}^1$.
 
(iv) By (i) and (ii), $f_5$ is unramified and each point in $f_2^{-1}(0,1,\infty)$ has ramification index a multiple of $n$ under $f_6$.

(v) The map $f_3$ is the projection $y$ composed with an automorphism of $\mathbb{P}^1$ which maps three 
branch points to 0,1 and $\infty$ with ramification indices 5. 

(vi) The map $f_4$ is the projection $y$ composed with an automorphism of $\mathbb{P}^1$ which maps three 
branch points to 0,1 and $\infty$ with ramification indices 5n.  

(vii) The curve $C_2$ is a compositum of  $C_1$ and  $\mathsf{C}_{5n}$
over $\mathbb{P}^1$ (via map $f_3\circ f_6$ and $f_4)$.

(viii) From (iv), (v) and (vi) and Abhyankar's lemma, we see that
$f_8$ is unramified. 

(ix) By (iv) and (viii), $C_2$ is an unramified cover of $\mathsf{C}_{2^{11}\cdot 3^3 \cdot n}$ which
maps surjectively 
onto $\mathsf{C}_{5n}$.

\end{proof}

\begin{cor}
If $n \geq 5$ and the only prime divisors of n are 2, 3 or 5, 
then: $\mathsf{C}_6 \Rightarrow \mathsf{C}_n$. 
\end{cor}

\begin{proof}
Write n as $2^r3^s5^t$ and $m$ as $2^r3^s$, we have: 

If $t=0$, this follows from Corollary 8.

If $t \neq 0$, then: (repeat using Proposition 10)
\begin{equation*}
\mathsf{C}_6 \Rightarrow \mathsf{C}_{2^{11t}\cdot 3^{3t}\cdot m} \Rightarrow
\mathsf{C}_{2^{11(t-1)}\cdot 3^{3(t-1)}\cdot 5m} \Rightarrow
\mathsf{C}_{2^{11(t-2)}\cdot 3^{3(t-2)}\cdot 5^2m} \Rightarrow ... \Rightarrow
\mathsf{C}_{2^{11(t-t)}\cdot 3^{3(t-t)}\cdot 5^tm}=\mathsf{C}_n.
\end{equation*}
\end{proof}

\noindent {\bf{Proof of Theorem 2:}} 
Assume the only prime divisors of $n$ and $m$ are 2,3 or 5. By Proposition 5, $\mathsf{C}_n$ lies over $\mathsf{C}_6$. By Corollary 11, $\mathsf{C}_6$ also lies over $\mathsf{C}_m$ and consequently $\mathsf{C}_n$ lies over $\mathsf{C}_m$. Similarly $\mathsf{C}_m$
also lies over $\mathsf{C}_n$. Hence $\mathsf{C}_n$ and $\mathsf{C}_m$ are equilvalent.
For the second part, just note that for $k \geq 5$, $\mathsf{C}_k$ is a hyperbolic hyperelliptic curve.
\hfill\(\blacksquare\)

\begin{prop}
$\mathsf{C}_{6\cdot 13} \Rightarrow \mathsf{C}_7$.
\end{prop}

\begin{proof}
Consider the following maps:
\[
\xymatrix{ \mathsf{C}_7\ar[r]^{h_1}&  \mathbb{P}^1 \ar[r]^{h_2}& \mathbb{P}^1 \ar[r]^{h_3}
& \mathbb{P}^1\ar[r]^{h_4}& \mathbb{P}^1 \ar[r]^{h_5}& \mathbb{P}^1 \ar[r]^{h_6}&  \mathbb{P}^1 \\}
\]

\

Here:

(i) The map $h_1$ is the degree 2 projection.

Bran($h_1$)=($1,t,t^2,t^3,t^4,t^5,t^6\infty$) and all local ramification indices are 2;

(ii) The map $h_2$ is $z+\frac{1}{z}$.

Ram($h_2$)=($1,-1$) with all ramification indices 2 and

Bran($h_2$)$\cup f_1$(Bran($h_1$))=($2,-2,t+t^6,t^2+t^5,t^3+t^4,\infty$).
This set is denoted by $D_2$;

(iii) The map $h_3$ is $\frac{z+2}{z-2}$.

$h_3$ is unramified and $h_3(D_2)=(\infty, 0, t_1, t_2, t_3, 1)$.
This set is denoted by $D_3$.

Here $t_i$ are roots of $7z^3+35z^2+21z+1=0$;

(iv) The map $h_4$ is $7z^3+35z^2+21z+1$.

Ram($h_4$)=(-$\frac{1}{3},-3,\infty$) with all ramification indices 2 or 3 and

Bran($h_4$) $\cup h_4(D_3)=(0,1,64,-\frac{64}{27},\infty)$. 
This set is denoted by $D_4$;

(v) The map $h_5$ is $256\cdot \frac{z-1}{z-64}$.

Clearly it is unramified and $h_5(D_4)=(0,4,13,256,\infty)$. 
This set is denoted by $D_5$;

(vi) The map $h_6$ is 
\begin{equation*}
\frac{z^{12301875}\cdot (z-6)^{32752512}\cdot (z-256)^{13}}
{(z-4)^{42120000}\cdot (z-13)^{2560000}\cdot (z+14)^{374400}}.
\end{equation*}

(This map is coming from a search using Belyi's formula (See Definition 21 and the proof of Proposition 23).)

Ram($h_6$)=($0,4,6,13,-14,256,\infty$) with corresponding ramification indices 

$3^95^4, 2^63^45^413, 2^73^913, 2^{12}5^4, 2^73^25^213,13,5$ and Bran($h_6$) $\cup h_6(D_5)=(0,1,\infty)$.

\

By (i)-(vi), $h_6 \circ h_5 \circ h_4 \circ h_3 \circ h_2 \circ h_1$ is a 
Belyi map with all local ramification indices dividing $2^{15}3^{10}5^413$ .
By Abhyankar's Lemma, Propostition 7 and Proposition 10, we have:

\begin{equation*}
\mathsf{C}_{6\cdot 13} \Rightarrow \mathsf{C}_{2^{15}3^{10}5^413} \Rightarrow \mathsf{C}_7
\end{equation*}

\end{proof}

\begin{prop}
$\mathsf{C}_{2^{16}\cdot 3^{10} \cdot 5^4 \cdot 13n} \Rightarrow \mathsf{C}_{7n}$
for $n \geq 1$.
\end{prop}

\begin{proof}
Let $h= h_6 \circ h_5 \circ h_4 \circ h_3 \circ h_2 \circ h_1$ where
$h_i$ are the maps in the last proposition. Note that $h$ is a Belyi map from $\mathsf{C}_7$ to $\mathbb{P}^1$ with all local ramification indices divides $2^{15}3^{10}5^413$.

\

Now consider the following diagram:

\[
\xymatrix{  \mathsf{C}_{2^{16}\cdot 3^{10} \cdot 5^4 \cdot 13n} \ar[d]^{f_1}& C_1  \ar[l]^{f_5} \ar[d]^{f_6}
& C_2 \ar[l]^{f_8} \ar[rd]^{f_7}\\
     \mathbb{P}^1 & \mathsf{C}_{7} \ar[l]^{f_2} \ar[r]^{f_3}& \mathbb{P}^1 & \mathsf{C}_{7n} \ar[l]^{f_4}\\}
\]

\

In this diagram:

(i) The map $f_1$ is the projection $y$ composed with an automorphism of $\mathbb{P}^1$ which maps three 
branch points to 0,1 and $\infty$ such that $f_1^{-1}(0)$ and $f_1^{-1}(1)$ each contains one point
with ramification index $2^{16}\cdot 3^{10} \cdot 5^4 \cdot 13n$ and $f_1^{-1}(\infty)$ contains two points with ramification indice $2^{15}\cdot 3^{10} \cdot 5^4 \cdot 13n$;

(ii) The map $f_2$ is the map $h$ above;

(iii) The curve $C_1$ is a compositum of  $\mathsf{C}_{2^{16}\cdot 3^{10} \cdot 5^4 \cdot 13n}$ and  $\mathsf{C}_{7}$
over $\mathbb{P}^1$;
 
(iv) By (i) and (ii), $f_5$ is unramified and each point in $f_2^{-1}(0,1,\infty)$ has ramification index a multiple of $n$ under $f_6$;

(v) The map $f_3$ is the projection $y$ composed with an automorphism of $\mathbb{P}^1$ which maps three 
branch points to 0,1 and $\infty$ with ramification indices 7;

(vi) The map $f_4$ is the projection $y$ composed with an automorphism of $\mathbb{P}^1$ which maps three 
branch points to 0,1 and $\infty$ with ramification indices 7n.;

(vii) The curve $C_2$ is a compositum of  $C_1$ and  $\mathsf{C}_{7n}$
over $\mathbb{P}^1$ (via map $f_3\circ f_6$ and $f_4)$;

(viii) By the computations in (iv), (v), (vi) and Abhyankar's lemma,
$f_8$ is unramified;

(ix) By (iv) and (viii), $C_2$ is an unramified cover of $\mathsf{C}_{2^{16}\cdot 3^{10} \cdot 5^4 \cdot 13n}$ which
maps subjectively onto $\mathsf{C}_{7n}$.

\end{proof}

\noindent {\bf{Proof of Theorem 4:}} 

Set $m=2^a3^b5^c$.

If $d=0$, this follows from Theorem 2.

If $d \neq 0$, then:
\begin{equation*}
\mathsf{C}_{6\cdot 13^d} \Rightarrow \mathsf{C}_{2^{16d}\cdot 3^{10d}\cdot 5^{4d}\cdot 13^d\cdot m} \Rightarrow
\mathsf{C}_{2^{16(d-1)}\cdot 3^{10(d-1)}\cdot 5^{4(d-1)} \cdot 13^{d-1}\cdot 7\cdot m}
\end{equation*}
\begin{equation*}
\Rightarrow
\mathsf{C}_{2^{16(d-2)}\cdot 3^{10(d-2)}\cdot 5^{4(d-2)} \cdot 13^{d-2}\cdot 7^2 \cdot m} 
\Rightarrow ... \Rightarrow
\mathsf{C}_{2^{16(d-d)}\cdot 3^{10(d-d)}\cdot 5^{4(d-d)} \cdot 13^{d-d}\cdot 7^d\cdot m}=\mathsf{C}_n.
\end{equation*}

\hfill\(\blacksquare\)

By similar construction as in Proposition 12, we can also have:

\begin{prop}
$\mathsf{C}_{6\cdot 11\cdot 43} \Rightarrow \mathsf{C}_7$.
\end{prop}

\begin{proof}
We consider the same maps as in Proposition 12 except that
we replace $h_6$ by:
\begin{equation*}
\frac{z^{8620425}\cdot (z-13)^{7208960}\cdot (z-56)^{1539648}}
{(z-4)^{14860800}\cdot (z-48)^{2507760}\cdot (z-256)^{473}}
\end{equation*}
we have:

Ram($h_6$)=($0,4,13,48,56,256,\infty$) with corresponding ramification indices 

$3^6\cdot5^2\cdot11\cdot43, 2^9\cdot3^3\cdot5^2\cdot43, 2^{17}\cdot5\cdot11, 
2^4\cdot3^6\cdot5\cdot43, 2^6\cdot3^7\cdot11,11\cdot43,5$ 

and Bran($h_6$) $\cup h_6(D_5)=(0,1,\infty)$.

Thus $h_6 \circ h_5 \circ h_4 \circ h_3 \circ h_2 \circ h_1$ is a 
Belyi map with all local ramification indices dividing $2^{18}\cdot3^{8}\cdot5^2\cdot11\cdot43$ .
By Abhyankar's lemma, Propostition 7 and Proposition 10, we have:

\begin{equation*}
\mathsf{C}_{6\cdot 11\cdot 43} \Rightarrow \mathsf{C}_{2^{18}\cdot3^{8}\cdot5^2\cdot11\cdot 43} \Rightarrow \mathsf{C}_7
\end{equation*}

\end{proof}

\begin{defn}
Let $k$ be a field. Let $P$ be a subset of natural numbers and S be a subset of points on
$\mathbb{P}^1(\bar{k})$. We call a curve $C$ is $P$-ramified over S if there
exists a morphism from $C$ to $\mathbb{P}^1$ such that all branch points are
contained in $S$ and all local ramification indices are contained in $P$. Given two subsets $S_1$
and $S_2$ of points on $\mathbb{P}^1(\bar{k})$, we say $S_1$ can be $P$-contracted to $S_2$,
if there exists a morphism $f: \mathbb{P}^1 \to \mathbb{P}^1$ such that $f(S_1)$ and Bran$(f)$ is contained in $S_2$ and all local ramification indices are contained in $P$.
\end{defn}

\begin{thm}
Let $k$=$\mathbb{Q}$. If a curve $C$ is $P$-ramified over $S$ which can be $P$-contracted
to $(0,1,\infty)$ such that all numbers in $P$ only have prime divisors 2,3 or 5, then $\mathsf{C}_6\Rightarrow C$.
If we further allow $7$ appearing as prime divisors of numbers in $P$, then there exists a positive integer $n$ such that
$\mathsf{C}_{6\cdot13^n} \Rightarrow C$.
\end{thm}

\begin{proof}
This follows from Theorem 2 and Theorem 3.
\end{proof}

\begin{prop}
If $\mathsf{C}_6\Rightarrow \mathsf{C}_n$ holds for any positive integer $n$, then for any curve $C$,
we have $\mathsf{C}_6\Rightarrow C$.
\end{prop}

\begin{proof}
By Belyi's theorem $C$ is $P$-ramified over $(0,1,\infty)$ for some finite set $P$.
Let $n$ be the least common multiple of numbers in $P$. Then we have:
$\mathsf{C}_6 \Rightarrow \mathsf{C}_n \Rightarrow C$.
\end{proof}

\medskip{}
In \cite{3}, we have the following conjecture:
\begin{namedthm}{Conjecture}
Let $C$ be any curve over $\bar{\mathbb{Q}}$. Then $\mathsf{C}_6\Rightarrow C$.
\end{namedthm}

We will describe a possible way to approach this conjecture in the last section.

\

\section{\label{section 3} Contraction of points on $\mathbb{P}^1(\bar{\mathbb{Q}})$}

In this section, we discuss the problem of contraction of points on $\mathbb{P}^1(\bar{\mathbb{Q}})$
with certain control on local ramification indices.

\medskip{}

The first result is from \cite{3}, theorem 4.4:

\medskip{}

\begin{thm}
Let $S$ be a finite set of points on $\mathbb{P}^1(\bar{\mathbb{Q}})$.
Then there exists a map 
\begin{equation*}
f: \mathbb{P}^1\rightarrow \mathbb{P}^1
\end{equation*}
 which is defined over $\mathbb{Q}$ such that:
\begin{equation*}
f(S)\cup \text{Ram}(f) \subset \mathbb{P}^1(\mathbb{Q})
\end{equation*}
and moreover, all local ramification indices are powers of 2.
\end{thm}

\medskip{}

Here, we will give a simplified proof:

\begin{proof}
Denote $m=$max(deg($s$)) for $s\in S$ and assume
$x\in S$ has degree $m$.

\medskip{}

Assume $2^{k-1}\leq m < 2^k$ for some positive integer $k$. Let $r=2^k-m$ and consider polynomials
$f\cdot g_r$ where $f$ is the minimal polynomial of $x$ and $g_r$
runs over all monic polynomials of degree $r$ with rational coefficients.
Let us denote by $L_r$ the space of such polynomials.

\medskip{}

We claim that there is a polynomial $g\in L_r$ such that all finite ramification points
of $F=fg$ are simple (order 2) and there are at least $r$ rational ramification points.
Indeed, given $x_1, ... , x_r\in \mathbb{Q}$, the condition that $x_1, ... ,x_r$ are ramification
points of $F$ yields a system of $r$ linear equations on the coefficients of $g_r$ in terms
of $x_i$ and the coefficients of $f$. The corresponding system of linear equations is nondegenerate
if $\{x_1, ... ,x_r\}$ does not intersect the common roots of $f'$ and $f$. Thus we obtain a rational map defined over $\mathbb{Q}$ from $\mathbb{A}^r$ to $L_r$ and clearly each point in $L_r$ only
has a finite number of preimages. Note that a condition that for $h\in L_r$ the derivative $h'$ has
multiple roots defines a divisor $D$ in $L_r$. Therefore the preimage of $D$ can not be the whole
domain of our rational map and thus we can pick some $(x_1,...,x_r)$ such that the corresponding
$F$ satisfying our condition.

\medskip{} 

Now by our claim we can pick one such polynomial $g$ and look at the map
$F: \mathbb{P}^1\mapsto \mathbb{P}^1$ given by $F=fg$. Note that the set of ramfication points
of $F$ consists of $r$ rational points, and some other points with algebraic degree less than $m$
and the point $\infty$. Also all ramification points except $\infty$ are simple and the ramification index at
$\infty$ is $2^k$. Thus, every point in the set $S\cup F(S)\cup Ram(F)$ has algebraic degree at most $m$, and the number of points with degree $m$ in $F(S)\cup Ram(F)$ is strictly less than that for $S$. Repeating this construction, we see that the composition
of all these maps is a desired map.

\end{proof}

\

Since every curve admits a map to $\mathbb{P}^1$ with simple ramification points, we have an immediate corollary:

\begin{cor}
Let $C$ be a curve over $\bar{\mathbb{Q}}$.
Then $C$ is $P-$ramified over a finite set of points on $\mathbb{P}^1(\mathbb{Q})$ with $P$ being the subset of natural numbers containing all powers of 2.
\end{cor}

\medskip{}

This theorem and its corollary is a generalization of the first step in the proof of Belyi's theorem in \cite{1}.
It is natural to consider whether in the second step in the proof of Belyi's theorem, one can also impose some restriction on local ramification indices. Let us consider the case of using Belyi's functions.

\begin{defn}
We call a morphism $f: \mathbb{P}^1 \rightarrow \mathbb{P}^1$ is a Belyi function with respect to a 
$k$-tuple $(n_1, ... ,n_k)$ if:
\begin{equation}
f(x)=\prod_{i=1}^{k}(x-n_i)^{r_i}
\end{equation}
with
\begin{equation*}
\text{Ram}(f)=(n_1,...,n_k,\infty) \ \ \text{and} \ \ f(\infty)=1.
\end{equation*}
\end{defn}

\medskip{}
\begin{namedthm}{Remark}
\textup{These maps are those appearing in Belyi's second proof of his theorem in \cite{2}.
Note that for $k\geq 3$, $\infty$ is a ramification point with index $k-1$.}
\end{namedthm} 

\medskip{}
A simple observation is:

\begin{prop}
Let $P$ be the subset of natural numbers whose prime divisors
are contained in a finite set of primes $\{p_1, ... , p_s\}$. Let $S$ be a finite set of integers $\{n_1, ... , n_k\}$ plus $\infty$ such that:

(i) for any pair $(i,j)$, $n_i-n_j\in P$;

(ii) $k-1 \in P$.

Then $S$ can be $P$-contracted to $(0,1,\infty)$.

\end{prop}
\begin{proof}
As in \cite{2}, in (3.1) let us take:
\begin{equation*}
r_i=(-1)^{i-1}V(n_1, ... ,\hat{n_i}, ... ,n_k)
\end{equation*}
where the term with a hat is to be omitted and $V$ denotes the Vandermonde determinant.
\end{proof}

\medskip{}

Conversely, if we use Belyi's functions to contract points, then the converse of the above
proposition is true for $k=3$:

\medskip{}
\begin{prop}
 Let $P$ be the subset of natural numbers whose prime divisors
are contained in primes $\{2, p_2, ... , p_s\}$. Let $S=\{n_1,n_2,n_3,\infty\}$. If $S$ is $P$-contracted to $(0,1,\infty)$
by some Belyi function, then for any pair $(i,j)$, we have: $n_i-n_j\in P$.

Moreover, there are only finitely many such sets $S$ modulo translation and multiplication.

\end{prop}

\begin{proof}
Let $f$ be a Belyi function with respect to $(n_1,n_2,n_3)$:
\begin{equation*}
f(x)=\prod_{i=1}^{3} (x-n_i)^{r_i}.
\end{equation*}

We have:
\begin{equation}
r_1+r_2+r_3=0
\end{equation}
and
\begin{equation}
(n_2+n_3)r_1+(n_1+n_3)r_2+(n_1+n_2)r_3=0
\end{equation}

\noindent with
\begin{equation*}
r_i \in P.
\end{equation*}

Modulo translation and multiplication, we may assume $n_1=0$ and $(n_2,n_3)$=1. From (3.2) and (3.3),
we have:
\begin{equation*}
n_2=r_3, n_3=-r_2, r_1=n_3-n_2 \ \text{and} \ (r_2,r_3)=1.
\end{equation*}

Hence, $n_2$, $n_3$ and $n_2-n_3$ are all in $P$.

Moreover, since (3.2) can be transformed into a unit equation in $\{2,p_2,...,p_s\}-$units, it only has
finitely many coprime solutions which means such $S$ are finite modulo translation and multiplication. (See Theorem 7.4.2 in \cite{7})
\end{proof}

\medskip{}

\begin{namedthm}{Remark}
\textup{
From this proposition, we see that in the case of $k=3$ if we use Belyi functions to contract points on $\mathbb{P}^1(\mathbb{Z})$, then the prime divisors of local ramification indices are depended on the prime divisors of pairwise differences between these points.
}
\end{namedthm}

\medskip{}
However, starting with $k=4$, we have exceptional examples. Let us see one example:

\medskip{}
\begin{namedthm}{Example}
Let $P$ be the subset of natural numbers whose prime divisors
are contained in $\{2,3\}$ and $S=\{0,1,5,6\}$. Then we have the following Belyi function
with respect to this 4-tuple:

\begin{equation*}
f(x)=\frac{(x-1)^3(x-6)^2}{x^2(x-5)^3}.
\end{equation*}

\noindent Hence, $S$ can be $P-$contracted to $(0,1,\infty)$ but $5$, which is the difference between 5 and 0, is not in $P$.

\end{namedthm}

\medskip{}
Although for $k\geq 4$ there are some exceptional examples, we have the following:

\medskip{}
\begin{thm}
Let $P$ be a subset of natural numbers containing prime divisors
$3, p_2, ... , p_s$. Then the set of collections of 4-tuples $(n_1, n_2, n_3 ,n_4)$ plus $\infty$ which can be $P$-contracted
to $(0,1,\infty)$ by some Belyi's functions are contained in some finite union of hyperplanes in
$\mathbb{A}^4(\mathbb{Z})$. (Modulo translation and multiplication, it's contained in some finite union of
lines in $\mathbb{A}^2(\mathbb{Q})$) Moreover, the number of such 4-tuples which do not satisfy condition (i) in Proposition 23 is infinite modulo translation
and multiplication.

\end{thm}

\begin{proof}
Let $f$ be a Belyi function with respect to the 4-tuple $(n_1,n_2,n_3,n_4)$:
\begin{equation*}
f(x)=\prod_{i=1}^{4} (x-n_i)^{r_i}.
\end{equation*}

Then we have:
\begin{equation}
r_1+r_2+r_3+r_4=0
\end{equation}

and
\begin{equation}
(n_2+n_3+n_4)r_1+(n_1+n_3+n_4)r_2+(n_1+n_2+n_4)r_3+(n_1+n_2+n_3)r_4=0
\end{equation}

and
\begin{equation}
(n_2n_3+n_2n_4+n_3n_4)r_1+...+(n_1n_2+n_1n_3+n_2n_3)r_4=0
\end{equation}

with
\begin{equation*}
r_i\in P.
\end{equation*}

Since (3.4) can be transformed into a unit equation in $\{3, p_2, ... ,p_s\}$-units, we have:

Either some proper subsum of $r_1+r_2+r_3+r_4$ vanishes or it will only have finitely many
coprime solutions. For each solution in the second case, the corresponding 4-tuple is contained in the hyperplane defined by (3.5) (Although such corresponding 4-tuple may not exist). The remaining
case is either $r_1+r_2, r_1+r_3$ or $r_1+r_4$ vanishes. Without loss of generality, assume $r_1+r_2=0$ which implies $r_3+r_4=0$. Thus, (3.5) and (3.6) are reduced to:
\begin{equation}
(n_2-n_1)r_1+(n_4-n_3)r_3=0
\end{equation}

and

\begin{equation}
(n_2-n_1)(n_3+n_4)r_1+(n_4-n_3)(n_1+n_2)r_3=0.
\end{equation}

Substitute (3.7) into (3.8) yields:
\begin{equation*}
(n_2-n_1)(-n_1-n_2+n_3+n_4)=0
\end{equation*} 
which means our 4-tuple is contained in the hyperplane defined by the equation:
\begin{equation*}
-x_1-x_2+x_3+x_4=0
\end{equation*} 

Moreover, from (3.7) and (3.8), if we translate $n_1$ to 0, all solutions of (3.7) and (3.8)
are: (modulo translation and multiplication)
\begin{equation*}
n_1=0, n_2=2r_3, n_3=r_1+r_3, n_4=r_3-r_1.
\end{equation*}
with $(r_1,r_3)=1$ and all prime divisors of them are in $\{3, p_2, ... ,p_s\}$.

Therefore, we have infinitely many such 4-tuples which do not satisfy condition (i) of Proposition 23 since
the unit equation:
\begin{equation*}
\frac{n_3}{n_2}+\frac{n_4}{n_2}=1
\end{equation*}
in $\{3, p_2, ..., p_s\}$-units only have finitely many coprime solutions.
\end{proof}

\medskip{}

\begin{namedthm}{Remark}
\textup{
By similar argument, we can get similar results for $k\geq 5$. Thus,
most $k$-tuples plus $\infty$ can not be $P$-contracted to $(0,1,\infty)$ by using Belyi's functions if we let $P$ be a subset
of natural numbers whose prime divisors lie in a finite set of primes. This suggests that Question 1.4 in \cite{4} may not have an affirmative answer. 
}
\end{namedthm}

\medskip{}

Now let us discuss using elliptic curves to contract points and their relation to our
unramified curve correspondence problem.

\medskip{}
Following \cite{3}:
\begin{namedthm}{Notation}
Let $E$ and $E'$ be two elliptic curves and $\pi$ and $\pi'$ be the
standard projection to $\mathbb{P}^1$. Write:
\begin{equation*}
E \rightharpoondown E'
\end{equation*}
if Bran$(\pi')$ is projectively equivalent to a set of four points in $\pi(E[\infty])$. Here,
$E[\infty]$ is the set of torsion points on $E$.
\end{namedthm}

One of the reasons why we study such relations comes from:

\begin{thm}
Let $C'$ be a hyperbolic curve and $g: C'\rightarrow \mathbb{P}^1$ be a morphism with
\begin{equation*}
\text{Bran}(g)\subset \pi(E_n[\infty])
\end{equation*}
for some elliptic curve $E_n$. Denote by $L$ the least common
multiple of all local ramification indices of $g$. Assume we have:
\begin{equation*}
E_0\rightharpoondown E_1 \rightharpoondown ... \rightharpoondown E_n
\end{equation*}
and let $C$ be a hyperbolic curve which admits a map onto $E_0$ such that there exists
one branch point whose all local ramification indices are divisible by $2^nL$. Then we have:
\begin{equation*}
C\Rightarrow C'.
\end{equation*}
\end{thm}

\begin{proof}
Let us prove for the case $n=1$. For $n>1$, the proof is similar.
Consider the following diagram:

\[
\xymatrix{C \ar[d]^{f_3} & C_1 \ar[l]^{f_1} \ar[d]^{f_4} & C_2
\ar[l]^{f_9} \ar[dd]^{f_{10}} &C_3 \ar[l]^{f_{11}} \ar[dd]^{f_{12}}
& C_4 \ar[l]^{f_{13}} \ar[ddd]^{f_{14}}\\
E_0 & E_0 \ar[l]^{f_2} \ar[d]^{f_5} \\
 & \mathbb{P}^1 & E_1 \ar[l]^{f_6} & E_1 \ar[l]^{f_7} \ar[d]^{f_8} \\
 & & & \mathbb{P}^1 & C' \ar[l]^{f_{15}}\\}
\]

\

In this diagram:

(i) The map $f_{15}$ is $g$;

(ii) The maps $f_5$, $f_6$ and $f_8$ are the degree 2 projections
such that: 

Bran$(f_{15})\subset f_8(E_1[\infty])$, Bran$(f_6)\subset f_5(E_0[\infty])$
;

(iii) The map $f_7$ is multiplication-by-$m$ map with $f_8^{-1}(\text{Bran}(f_{15})) \subset E_1[m]$;

(iv) The map $f_2$ is multiplication-by-$n$ map with $f_5^{-1}(\text{Bran}(f_6))\subset E_0[n]$.;

(v) The map $f_3$ is a map onto $E_0$ branched at the identity element of $E_0$ with all local
ramification indices being divisible by 2$L$;
 
(vi) The curve $C_1$ is a compositum of $C$ and $E_0$. By (iv) and (v),
$f_1$ is unramified and points in $f_5^{-1}(\text{Bran}(f_6))$ have
local ramification indices $2L$ under $f_4$;

(vii) The curve $C_2$ is a compositum of $C_1$ and $E_1$.
By (vi), $f_9$ is unramified and the local ramification index of the identity element of
$E_1$ under $f_{10}$ is divisible by $L$;

(viii) The curve $C_3$ is a compositum of $C_2$ and $E_1$.
Clearly $f_{11}$ is unramified and by (vii) the local ramification indices of points in 
$f_8^{-1}(\text{Bran}(f_{15}))$ under $f_{12}$ are divisible by $L$
;

(ix) The curve $C_4$ is a compositum of $C_3$ and $C'$. From the computation
in (viii), we see that $f_{13}$ is unramified.

\medskip{}
By (vi)-(ix), we see that $C_4$ is an unramified cover of $C$ which maps
onto $C'$ and consequently we have:

\begin{equation*}
C\Rightarrow C'.
\end{equation*}
\end{proof}

\medskip{}
\begin{defn}
Given a finite set $S$ of points on $\mathbb{P}^1(\mathbb{\bar{Q}})$, we call
$S$ can be contracted to $(a,b,c,d)$ if there exist some elliptic curves $E_0=E(a,b,c,d), E_1, ... ,E_n$ with:
\begin{equation*}
E_0\rightharpoondown E_1 \rightharpoondown ... \rightharpoondown E_n
\end{equation*}
such that $S$ is projectively equilvalent to a subset in $\pi (E_n[\infty])$. Here, $\pi$ is
the standard projection of $E_n$ to $\mathbb{P}^1$.
\end{defn}

\medskip{}
\begin{thm}
Let $C$ be a curve, $p$ be an odd prime and $P$ be the subset of natural numbers whose prime divisors are less than $p$. If $C$ can be $P$-contracted to a finite set of points $S$
which can be contracted to $(a,b,c,d)$ which can be $P$-contracted to $(0,1,\infty)$, then there exists
$n\in P$ such that:
\begin{equation*}
\mathsf{C}_n\Rightarrow C.
\end{equation*}
\end{thm}

\begin{proof}
By assumption, there exists $L\in P$ and a map:
\begin{equation*}
f: C\rightarrow \mathbb{P}^1 \ \ \text{with} \ \  \text{Bran}(f)\subset S
\end{equation*}
such that all local ramification indices divide $L$.

Also there exists $M\in P$ and a map:
\begin{equation*}
g: E(a,b,c,d)=E \rightarrow \mathbb{P}^1 \ \ \text{with} \ \ \text{Bran}(g)\subset (0,1,\infty)
\end{equation*}
such that all local ramification indices of $g$ divide $M$.

Now let us consider the following diagram:
\[
\xymatrix{ \mathsf{C}_{2LM} \ar[d]^{f_1} & C_1 \ar[l]^{f_2} \ar[d]^{f_3} \\
\mathbb{P}^1 & E \ar[l]^{f_4} \\
}
\]

In this diagram:

(i) The map $f_1$ is the standard degree 2 projection combined with an automorphism of $\mathbb{P}^1$
such that $(0,1,\infty)$ are contained in the branch locus of $f_1$; 

(ii) The map $f_4$ is the map $g$;

(iii) The curve $C_1$ is a compositum of $\mathsf{C}_{2LM}$ and $E$ via $f_1$ and $f_4$.
From (i) and (ii) we see that $f_2$ is unramified and $f_3$ is a map from $C_1$ onto $E$ such that at
least one branch point have all local ramification indices $2L$. By Theorem 30, we are done.
\end{proof}

\medskip{}
A direct corollary is:

\medskip{}
\begin{cor}
Let $C$ be a curve and $P$ be the subset of natural numbers whose
prime divisors lie in $\{2,3,5\}$. If $C$ can be $P$-contracted to a finite set of points $S$ which can be contracted to $(a,b,c,d)$ which can be $P$-contracted to $(0,1,\infty)$, then we have:
\begin{equation*}
\mathsf{C}_6 \Rightarrow C.
\end{equation*}

\end{cor}

\medskip{}
\begin{namedthm}{Remark}
\textup{
From this corollary, we see that if we want to use elliptic curves to attack the unramified curve correspondence problem, one important thing is the intersection of the image under the standard projection of the torsion points for
two different elliptic curves on $\mathbb{P}^1$ as well as the intersection of the image under the standard projection of the torsion points for one elliptic curve and the set of roots of unity on $\mathbb{P}^1$. In general, the intersection number is always finite (see \cite{6}), but we only need to find some special elliptic curves to approach our problem.
}
\end{namedthm}

\

\section{\label{section 4} A possible procedure to approach conjecture 18}

From Proposition 17 and our proof of Theorem 2, Theorem 4 and Theorem 32, we propose a possible way
to approach Conjecture 18:

\

Step 0: We already know (by Theorem 2)that if $n\geq 5$ is a positive integer whose only prime divisors are 2,3 or 5, then we have:
\begin{equation*}
\mathsf{C}_6\Leftrightarrow \mathsf{C}_n.
\end{equation*}

\

Step I: Start with $p=7$.

\

Step II: Let us show that $\mathsf{C}_p$ is $P$-ramified over some points $S$ which can
be contracted to $(a,b,c,d)$ which can be $P$-contracted to $(0,1,\infty)$ (or more intermediate steps like these) such that all numbers in $P$ only have prime divisors less
than $p$ and deduce that:
\begin{equation*}
\mathsf{C}_6 \Rightarrow \mathsf{C}_p.
\end{equation*}

\

Step III: Use the construction in last step (which is a combination of diagrams in Proposition 10, Theorem 30 and Theorem 32) to show:
\begin{equation*}
\mathsf{C}_{mn} \Rightarrow \mathsf{C}_{pn}
\end{equation*}
for some $m$ whose prime divisors are less than $p$ and for any $n\geq 1$.

\

Step IV: Use the result in last step to show:
\begin{equation*}
\mathsf{C}_{6} \Rightarrow \mathsf{C}_{n}
\end{equation*}
for all $n$ whose prime divisors are less than or equal to p. By Proposition 5 we can conclude
that $\mathsf{C}_n$ and $\mathsf{C}_m$ are equivalent for any $n$ and $m$ whose prime
divisors are less than or equal to $p$.

\

Step V: Consider the next prime and go back to Step II.

\

If eventually we can finish the above procedure for all primes, then by Proposition 17, 
Conjecture 18 will be true.

\

Actually the only hard part of the above procedure is Step II.
Step III and Step IV can be done in a similar fashion as we did in Proposition 10, Corollary 11
Theorem 30 and Theorem 32.

\medskip{}

\begin{prop}
Suppose $\mathsf{C}_p$ is $P$-ramified over some points $S$ which can
be contracted to some 4-tuple $(a,b,c,d)$ which can be $P$-contracted to $(0,1,\infty)$ such that all numbers in $P$ only have prime divisors less
than $p$.  Then there exists some $m$ whose prime divisors are less than $p$ such that for any
$n\geq 1$, we have:
\begin{equation*}
\mathsf{C}_{mn} \Rightarrow \mathsf{C}_{pn}.
\end{equation*}
\end{prop}

\

\begin{prop}
Assume $\mathsf{C}_n$ and $\mathsf{C}_m$ are equilvalent
for any $n$ and $m$ whose prime divisors are less than $p$. Suppose
there exists some $m$ whose prime divisors are less than $p$ such that for any
$n\geq 1$, 
$\mathsf{C}_{mn} \Rightarrow \mathsf{C}_{pn}.$
Then we have:
\begin{equation*}
\mathsf{C}_6 \Rightarrow \mathsf{C}_n
\end{equation*}
for any $n$ whose prime divisors are less than or equal to $p$.

\end{prop}
\begin{proof} (Sketch)
As mentioned, it is similar as the proof of Proposition 10, Corollary 11, Theorem 30 and Theorem 32. For the proof of Proposition 35, we will use a diagram similar as in the proof of Proposition 10.
Replace $\mathsf{C}_5$ by $\mathsf{C}_p$ and $\mathsf{C}_{5n}$ by $\mathsf{C}_{pn}$. The maps  $f_3$ and $f_4$ are still the projection $y$ composed with an automorphism of $\mathbb{P}^1$ such that both of them have branch points $\{0,1,\infty\}$. The difference is in Proposition 10, $f_2$ is a map from $\mathsf{C}_5$ to $\mathbb{P}^1$. Here we do not have such a map.
Instead under our assumption, $f_2$ will be replaced by a diagram which is a combination of the diagrams in Theorem 30 and Theorem 32. Also we can find one desired positive integer $m$ as in the proof of Theorem 32. Now Proposition 35 will be established if we do the similar computation as in Theorem 30 and Theorem 32. For Proposition 36, we can prove it in the same way as the proof
of Corollary 11 (Instead of repeating using Proposition 10, this time we repeat using Proposition 35).

\end{proof}

\

\begin{namedthm}{Remark}
\textup{Finally, let us describe a directed graph structure between all hyperbolic curves.
 We regard each hyperbolic
curve as a point in our graph.
If $C_1$ and $C_2$ are two hyperbolic curves such that $C_1$ implies $C_2$, then we associate
a directed edge from $C_1$ to $C_2$. If they are equilvalent, then we associate a simple
edge between $C_1$ and $C_2$. In this way, Conjecture 18 can be formulated as:
This graph is strongly connected. Even if Conjecture 18 does not hold, it is still interesting
to investigate the structure of subsets of coprime number $m$ and $n$ with different domination
areas of $\mathsf{C}_m$ over $\mathsf{C}_n$ and also modular curves $X(n)$. Proposition 12 and 14 are two examples of this.}
\end{namedthm}

\

\textbf{Acknowledgments.} The first author was partially supported by the Russian Academic Excellence Project '5-100' and by Simons Travel Grant. The second author was supported by the
MacCracken Program offered by New York University.

\

\

\

Fedor Bogomolov\\
Courant Institute of Mathematical Sciences, New York University\\
251 Mercer Street, New York, NY 10012, USA\\
Email: bogomolo@cims.nyu.edu

\medskip{}
Also:\\
National Research University, Higher School of Economics, Russian Federation.

\medskip{}
Jin Qian\\
Courant Institute of Mathematical Sciences, New York University\\
251 Mercer Street, New York, NY 10012, USA\\
Email: jq333@nyu.edu
\end{document}